\documentclass[reqno,12pt]{amsart}
\usepackage[latin1]{inputenc}
\usepackage[english]{babel}
\usepackage{amsmath, amsthm, amssymb, amsopn, amsfonts, amstext, stmaryrd, enumerate, color, mathtools, hyperref, mathrsfs, enumitem, url}
\usepackage[final]{showkeys}
\usepackage[margin=1.1in]{geometry}
\numberwithin{equation}{section}

\hypersetup{
	colorlinks=true,
	linkcolor=blue,
	citecolor=red,
	urlcolor=black,
	linktoc=all
}

\long\def\/*#1*/{}

\newcommand{\C}{\mathscr{C}}
\newcommand{\E}{\mathcal{E}}

\renewcommand{\L}{\mathcal{L}}

\newcommand{\R}{\mathbb{R}}

\newcommand{\W}{\mathbb{W}}
\newcommand{\loc}{{\rm loc}}

\newcommand{\dist}{{\mbox{\normalfont dist}}}
\newcommand{\bequ}{\begin{equation}}
\newcommand{\nequ}{\end{equation}}
\newcommand{\PV}{\mbox{\normalfont P.V.}}

\newcommand{\DG}{\mbox{\normalfont DG}^{s,p}}
\newcommand{\DGp}{\mbox{\normalfont DG}_+^{s,p}}
\newcommand{\DGm}{\mbox{\normalfont DG}_-^{s,p}}
\newcommand{\DGpm}{\mbox{\normalfont DG}_\pm^{s,p}}
\newcommand{\wDG}{\widetilde{\mbox{\normalfont DG}}{\mathstrut}^{s,p}}
\newcommand{\wDGp}{\widetilde{\mbox{\normalfont DG}}{\mathstrut}_+^{s,p}}
\newcommand{\wDGm}{\widetilde{\mbox{\normalfont DG}}{\mathstrut}_-^{s,p}}
\newcommand{\wDGpm}{\widetilde{\mbox{\normalfont DG}}{\mathstrut}_\pm^{s,p}}

\DeclareMathOperator{\supp}{supp}

\DeclareMathOperator{\Tail}{Tail}

\def\Xint#1{\mathchoice
{\XXint\displaystyle\textstyle{#1}}%
{\XXint\textstyle\scriptstyle{#1}}%
{\XXint\scriptstyle\scriptscriptstyle{#1}}%
{\XXint\scriptscriptstyle\scriptscriptstyle{#1}}%
\!\int}
\def\XXint#1#2#3{{\setbox0=\hbox{$#1{#2#3}{\int}$ }
\vcenter{\hbox{$#2#3$ }}\kern-.6\wd0}}

\def\dashint{\Xint-}

\theoremstyle{plain}
\newtheorem{definition}{Definition}[section]
\newtheorem{theorem}[definition]{Theorem}
\newtheorem{proposition}[definition]{Proposition}
\newtheorem{lemma}[definition]{Lemma}
\newtheorem{corollary}[definition]{Corollary}

\theoremstyle{definition}

\newtheorem{question}{Question}

\renewcommand{\le}{\leqslant}

\renewcommand{\ge}{\geqslant}

\def\Xint#1{\mathchoice
{\XXint\displaystyle\textstyle{#1}}%
{\XXint\textstyle\scriptstyle{#1}}%
{\XXint\scriptstyle\scriptscriptstyle{#1}}%
{\XXint\scriptscriptstyle\scriptscriptstyle{#1}}%
\!\int}
\def\XXint#1#2#3{{\setbox0=\hbox{$#1{#2#3}{\int}$ }
\vcenter{\hbox{$#2#3$ }}\kern-.6\wd0}}

\def\dashint{\Xint-}

\title[Fractional De Giorgi classes and applications to regularity]{Fractional De Giorgi classes and applications to nonlocal regularity theory}

\author{Matteo Cozzi}

\address{
\newline
\textit{Matteo Cozzi}
\newline
BGSMath Barcelona Graduate School of Mathematics \& Universitat Polit\`ecnica de Catalunya, Departament de Matem\`{a}tiques, Diagonal 647, 08028 Barcelona, Spain
\newline
\textit{E-mail address}: \textit{\tt matteo.cozzi@upc.edu}
}

\thanks{This note is mostly based on a talk given by the author at a conference held in Bari on May 29-30, 2017, as part of the INdAM intensive period~``Contemporary research in elliptic PDEs and related topics''. The author wishes to thank Serena Dipierro, the Universit\`a degli Studi di Bari, and INdAM for their kind invitation, warm hospitality, and financial support. The author is also supported by the~``Mar\'ia de Maeztu'' MINECO grant~MDM-2014-0445 and by the MINECO grant~MTM2017-84214-C2-1-P}

\subjclass[2010]{49N60, 35B45, 35B50, 35B65, 35R11, 47G20}

\keywords{Fractional~De~Giorgi classes, nonlocal Caccioppoli inequality, H\"older continuity, Harnack inequality, Nonlocal functionals, nonlinear integral operators}

\begin{document}

\begin{abstract}
We present some recent results obtained by the author on the regularity of solutions to nonlocal variational problems. In particular, we review the notion of fractional De Giorgi class, explain its role in nonlocal regularity theory, and propose some open questions in the subject.
\end{abstract}

\maketitle

\section{Introduction}

De Giorgi classes are a powerful tool in the regularity theory of Partial Differential Equations and Calculus of Variations. By definition, their elements are functions that belong to a Sobolev space and satisfy Caccioppoli inequalities at all of their levels. Their introduction can be dated back to the fundamental work of~De~Giorgi~\cite{DeG57}, where he devised them to prove the H\"older continuity of solutions to second order equations in divergence form with bounded measurable coefficients. Later on, Giaquinta \& Giusti~\cite{GG82} discovered that~De~Giorgi classes could also be utilized to prove H\"older estimates for minimizers of non-differentiable functionals, one of the first general regularity results that did not make use of the Euler-Lagrange equation. A couple of years later, DiBenedetto \& Trudinger~\cite{DT84} showed that~De~Giorgi classes are not only responsible for continuity properties, but also lead to Harnack inequalities. See the classical books~\cite{LU68,Giu03}, and the more recent~\cite{GM12} for additional information.

The aim of this work is to review and further enrich the theory developed by the author in~\cite{C17}, about fractional notions of~De~Giorgi classes and their applications to the regularity properties of solutions to nonlocal variational problems.

We consider the class~$\wDG$ and its subclass~$\DG$, both made up by functions that are contained in a Sobolev space of fractional order and satisfy a family of nonlocal Caccioppoli-type estimates. The inequality defining~$\wDG$ has a somewhat similar structure to that of standard~De~Giorgi classes and it is by now fairly understood, thanks to a number of contributions available in the literature, such as~\cite{DKP16,KMS15a,BP16}. On the other hand, the inequality that corresponds to the class~$\DG$ is stronger and incorporates a purely nonlocal term that has no classical counterpart. To the best of our knowledge, this last inequality has been previously considered only by Caffarelli, Chan \& Vasseur~\cite{CCV11} in a nonlocal parabolic context.

Throughout Section~\ref{DeGsec} we state several results pertaining to these classes. In particular, we establish that:
\begin{enumerate}[label=$(\alph*)$,leftmargin=*]
\item the elements of the class~$\wDG$ (and, therefore, of its smaller subset~$\DG$) are locally bounded functions---see Theorem~\ref{ulocboundthm};
\item \label{holdpoint} the functions of~$\DG$ are locally uniformly H\"older continuous---see Theorem~\ref{ulocHolderthm};
\item \label{harpoint} Harnack-type inequalities are true for functions that belong to~$\DG$ and that are non-negative---see Theorems~\ref{DGharthm} and~\ref{DGweakHarthm}.
\end{enumerate}
Furthermore, in Appendix~\ref{charexapp} we show by means of an explicit example that, for some choices of the parameters~$s$ and~$p$, the results of points~\ref{holdpoint} and~\ref{harpoint} cannot be extended to the larger class~$\wDG$.

Sections~\ref{minsec} and~\ref{solsec} are devoted to applications in the regularity theory for nonlocal variational problems. In Section~\ref{minsec} we deal with minimizers of energy functionals obtained as the sum of a possibly non-differentiable potential and of an interaction term comparable to the Gagliardo seminorm of a fractional Sobolev space. By showing that these extrema are contained in the fractional~De~Giorgi class~$\DG$, we deduce their H\"older continuity and the validity of Harnack inequalities, thanks to the statements of Section~\ref{DeGsec}. In Section~\ref{solsec} we approach in a similar way the regularity of solutions to equations driven by singular integral operators, such as fractional Laplacians and other nonlinear variations. These results complement and extend several available contributions, as for instance~\cite{Kas09,Sil06,DKP14,DKP16}.

\section{Fractional De Giorgi classes} \label{DeGsec}

We begin by introducing the larger set~$\wDG$, which we will sometimes call~\emph{weak fractional~De~Giorgi class}. To do this, we first need to fix some terminology.

Unless otherwise stated, throughout the whole paper~$n \ge 1$ is an integer indicating the dimension of the Euclidean space under consideration,~$s \in (0, 1)$ is a parameter representing a fractional order of differentiability, and~$p > 1$ is an integrability exponent. Also,~$\Omega$ always denotes a bounded open subset of the space~$\R^n$.

With the symbol~$W^{s, p}(\Omega)$ we denote the fractional Sobolev space composed by those functions~$u$ that lie in the Lebesgue space~$L^p(\Omega)$ and have finite Gagliardo seminorm
$$
[u]_{W^{s, p}(\Omega)} := \left( \int_\Omega \int_\Omega \frac{|u(x) - u(y)|^p}{|x - y|^{n + s p}} \, dx dy \right)^{\frac{1}{p}}.
$$
As it is customary, we endow~$W^{s, p}(\Omega)$ with the norm~$\| \cdot \|_{W^{s, p}(\Omega)}$ defined by the identity~$\| u \|_{W^{s, p}(\Omega)}^p := \| u \|_{L^p(\Omega)}^p + [ u ]_{W^{s, p}(\Omega)}^p$ and we simply write~$H^s(\Omega) := W^{s, 2}(\Omega)$ when~$p = 2$.

Another functional space that we will often use is the weighted Lebesgue space~$L^{p- 1}_s(\R^n)$ made up by all measurable functions~$u: \R^n \to \R$ for which
$$
\int_{\R^n} \frac{|u(x)|^{p - 1}}{(1 + |x|)^{n + s p}} \, dx < +\infty.
$$
For~$u \in L^{p - 1}_s(\R^n)$, the quantities
\begin{equation} \label{taildef}
\Tail_{s, p}(u; x_0, R) := \left( R^{s p}\int_{\R^n \setminus B_R(x_0)} \frac{|u(x)|^{p - 1}}{|x - x_0|^{n + s p}} \, dx \right)^{\frac{1}{p - 1}}
\end{equation}
and~$\overline{\Tail}_{s, p}(u; x_0, R) := R^{- \frac{s p}{p - 1}} \Tail_{s, p}(u; x_0, R)$ are finite for every point~$x_0 \in \R^n$ and every radius~$R > 0$. The tail term~\eqref{taildef}---introduced in~\cite{DKP14,DKP16}---is conveniently used to describe the behavior of~$u$ far away from~$x_0$. When~$x_0$ is the origin of~$\R^n$, we just write~$B_R := B_R(0)$,~$\Tail_{s, p}(u; R) := \Tail_{s, p}(u; 0, R)$, and~$\overline{\Tail}_{s, p}(u; R) := \overline{\Tail}_{s, p}(u; 0, R)$.

For~$k \in \R$, we indicate the super- and sublevel sets of a function~$u: \R^n \to \R$ respectively with~$A^+(k)$ and~$A^-(k)$. In symbols,
$$
A^+(k) := \{ u > k\} \quad \mbox{and} \quad A^-(k) := \{ u < k \}.
$$
We also write~$A^\pm(k, x_0, R) := A^\pm(k) \cap B_R(x_0)$ for their intersections with the ball~$B_R(x_0)$ and, as before,~$A^\pm(k, R) := A^\pm(k, 0, R)$.

Finally,~$v_+ := \max \{ v, 0 \}$ and~$v_- := (-v)_+ = \max \{ -v, 0 \}$ indicate respectively the positive and negative parts of a function~$v$.

With this in hand, we can now state the definition of \emph{weak fractional~De~Giorgi class}.

\begin{definition}[\bf Weak fractional De Giorgi class $\wDG$] \label{wDGdef}
Let~$d, \lambda \ge 0$ and~$H \ge 1$. A function~$u \in L^{p - 1}_s(\R^n)$ with~$u|_\Omega \in W^{s, p}(\Omega)$ belongs to~$\wDGpm(\Omega; d, H, \lambda)$ if
\begin{equation} \label{wDG+-def}
\begin{aligned}
& [(u - k)_\pm]_{W^{s, p}(B_{r}(x_0))}^p \phantom{ + \int_{B_{r}(x_0)} (u(x) - k)_\pm \int_{B_{2 R_0}(x)} \frac{dy}{|x - y|^{n + s p}} \, dx } \\
& \hspace{30pt} \le H \, \Bigg\{ R^\lambda d^p |A^\pm(k, x_0, R)| + \frac{R^{(1 - s) p}}{(R - r)^p} \| (u - k)_\pm \|_{L^p(B_R(x_0))}^p \\
& \hspace{30pt} \quad + \frac{R^{n + s p}}{(R - r)^{n + s p}} \| (u - k)_\pm \|_{L^1(B_R(x_0))} \overline{\Tail}_{s, p}((u - k)_\pm; x_0, r)^{p - 1} \Bigg\}
\end{aligned}
\end{equation}
holds for every point~$x_0 \in \Omega$, radii~$0 < r < R < \dist(x_0, \partial \Omega)$, and level~$k \in \R$. In addition,~$u \in \wDG(\Omega; d, H, \lambda)$ if and only if~$u \in \wDGp(\Omega; d, H, \lambda) \cap \wDGm(\Omega; d, H, \lambda)$.
\end{definition}

According to~\eqref{wDG+-def}, functions in~$\wDG$ satisfy a fractional and nonlocal version of the usual Caccioppoli inequality at all levels~$k$. Broader definitions can be considered, along the lines of those of~\cite[Section~6]{C17}. Here, we preferred to keep things as simple as possible, in order to favor readability over generality. Of course, we could take into account an even simpler definition, by removing the last line of~\eqref{wDG+-def} and thus neglecting the presence of tail terms. This choice would certainly be more elegant, as then the class~$\wDG(\Omega; d, H, \lambda)$ would be a subset of the Sobolev space~$W^{s, p}(\Omega)$. However, this definition would be too restrictive in light of our applications in Section~\ref{minsec} and~\ref{solsec}, which ultimately motivate the structure of~\eqref{wDG+-def}.

As for their classical counterparts, prototypical examples of functions belonging to weak fractional~De~Giorgi classes are the solutions of elliptic equations. While for standard De Giorgi classes, these equations are second order PDEs, the ones that are naturally associated to~$\wDG$ are fractional order equations driven by singular integral operators, such as the fractional Laplacian and nonlinear variations. This connection has been already observed by many authors---see, e.g.,~\cite{DKP16,KMS15a,BP16}.

As it has been partially anticipated in the introduction, classical De Giorgi classes were introduced for their importance in the regularity theory for second order equations, as they encode virtually all the information concerning the basic regularity properties enjoyed by the solutions of such equations---namely, local boundedness, H\"older continuity, and the validity of Harnack inequalities. The first goal of this section is to discuss whether these properties continue to hold for the fractional class~$\wDG$

As a first observation, we show that the elements of~$\wDG$ are locally bounded functions. Of course, when~$n < sp$ their boundedness (and H\"older continuity) is guaranteed by the Morrey-type embedding~$W^{s, p} \hookrightarrow C^{s - n / p}$ (see, e.g.,~\cite[Theorem~8.2]{DPV12}). Hence, at least for what concerns regularity, we can restrict ourselves to dealing with the case of~$n \ge sp$. For the sake of a simpler exposition, we will in fact suppose throughout the whole section that~$n > sp$. We stress that the critical case~$n = sp$---which is excluded here---only poses few additional technical difficulties and can be treated similarly---see~\cite[Section~6]{C17}.

\begin{theorem}[\bf Local boundedness of functions in $\wDG$] \label{ulocboundthm}
Let~$u \in \wDG(\Omega; d, H, \lambda)$ for some~$d, \lambda \ge 0$ and~$H \ge 1$. Then,~$u \in L^\infty_\loc(\Omega)$ and there exists a constant~$C \ge 1$, depending only on~$n$,~$s$,~$p$, and~$H$, such that
$$
\| u \|_{L^\infty(B_R(x_0))} \le C \left\{ \left( \dashint_{B_{2 R}(x_0)} |u(x)|^p \, dx \right)^{\frac{1}{p}} + \Tail_{s, p}(u; x_0, R) + R^{\frac{\lambda + s p}{p}} d \right\}
$$
for every~$x_0 \in \Omega$ and~$0 < R < \dist \left( x_0, \partial \Omega \right) / 2$.
\end{theorem}

We observe that a different version of Theorem~\ref{ulocboundthm}, valid for variants of weak fractional~De~Giorgi classes that do not include the presence of a tail term on the right-hand side of~\eqref{wDG+-def} and for~$p = 1$, is contained in~\cite{M11}. 

The estimate of Theorem~\ref{ulocboundthm} follows from analogous one-sided bounds for the elements of~$\wDGp$ and~$\wDGm$. By symmetry, it suffices to prove the following result.

\begin{proposition} \label{ulocboundprop}
Let~$u \in \wDGp(\Omega; d, H, \lambda)$ for some~$d, \lambda \ge 0$ and~$H \ge 1$. Then, there exists a constant~$C \ge 1$, depending only on~$n$,~$s$,~$p$, and~$H$, such that
\begin{equation} \label{ulocbound}
\sup_{B_R(x_0)} u \le C \left\{ \left( \dashint_{B_{2 R}(x_0)} u_+(x)^p \, dx \right)^{\frac{1}{p}} + \Tail_{s, p}(u_+; x_0, R) + R^{\frac{\lambda + s p}{p}} d \right\}
\end{equation}
for every~$x_0 \in \Omega$ and~$0 < R < \dist \left( x_0, \partial \Omega \right) / 2$.
\end{proposition}
\begin{proof}
Our argument is a simple variation of the one that leads to, say,~\cite[Theorem~7.2]{Giu03}.

Up to a translation, we may assume that~$x_0$ is the origin. Let two radii~$R \le \rho < \tau \le 2 R$ be fixed, take~$k \ge 0$, and set~$w_k := (u - k)_+$. Using H\"older and fractional Sobolev inequalities, it is not hard to infer that
$$
\| w_k \|_{L^p(B_\rho)}^p \le C |A^+(k, \rho)|^{sp / n} \left( [w_k]_{W^{s, p}(B_{(\tau + \rho) / 2})}^p + \frac{\tau^{(1 - s) p}}{(\tau - \rho)^p} \| w_k \|_{L^p(B_\tau)}^p \right),
$$
for some constant~$C \ge 1$ depending only on~$n$,~$s$, and~$p$. This and~\eqref{wDG+-def} give
\begin{equation} \label{Linftytech}
\begin{aligned}
\| w_k \|_{L^p(B_\rho)}^p & \le C |A^+(k, \rho)|^{sp / n} \Bigg\{ \tau^\lambda d^p  |A^+(k, \tau)| + \frac{\tau^{(1 - s) p}}{(\tau - \rho)^p} \| w_k \|_{L^p(B_\tau)}^p \\
& \quad + \frac{\tau^{n + s p}}{(\tau - \rho)^{n + s p}} \| w_k \|_{L^1(B_\tau)} \overline{\Tail}_{s, p}(w_k; R)^{p - 1} \Bigg\},
\end{aligned}
\end{equation}
where~$C$ may now depend on~$H$ as well.

Letting~$0 \le h < k$, it is easy to see that
$$
|A^+(k, r)| \le \frac{\| w_h \|_{L^p(B_r)}^p}{(k - h)^p}, \quad \| w_k \|_{L^p(B_r)}^p \le \| w_h \|_{L^p(B_r)}^p, \quad \| w_k \|_{L^1(B_r)} \le \frac{\| w_h \|_{L^p(B_r)}^p}{(k - h)^{p - 1}},
$$
and
$$
\overline{\Tail}_{s, p}(w_k; r)^{p - 1} \le \overline{\Tail}_{s, p}(w_0; r)^{p - 1} = r^{- s p} \Tail_{s, p}(u_+; r)^{p - 1}
$$
for every~$r > 0$. Accordingly,~\eqref{Linftytech} yields the estimate
$$
\varphi(k, \rho) \le \frac{C \tau^{- s p}}{(k - h)^{\frac{s p^2}{n}}} \left\{ \frac{\tau^{\lambda + s p} d^p}{(k - h)^p} + \frac{\tau^{p}}{(\tau - \rho)^p}  + \frac{\tau^{n + s p} \Tail_{s, p}(u_+; R)^{p - 1}}{(\tau - \rho)^{n + s p}(k - h)^{p - 1}} \right\} \varphi(h, \tau)^{1 + \frac{s p}{n}}
$$
for the quantities~$\varphi(\ell, r) := \| w_\ell \|_{L^p(B_r)}^p$.

Consider now the sequences~$\{ k_i \}$ and~$\{ \rho_i \}$, respectively defined by~$k_i := M (1 - 2^{- i})$ and~$\rho_i := (1 + 2^{- i}) R$ for all integers~$i \ge 0$ and for some~$M > 0$ to be chosen later. Set~$\varphi_i := \varphi(k_i, \rho_i)$. By evaluating the last inequality along these two sequences, we obtain
$$
\varphi_{i + 1} \le \frac{C \, 2^{(n + 3 p) i}}{M^{s p^2 / n} R^{s p}} \left\{ \frac{R^{\lambda + s p} d^p}{M^p} + 1 + \frac{\Tail_{s, p}(u_+; R)^{p - 1}}{M^{p - 1}} \right\} \varphi_i^{1 + \frac{s p}{n}}.
$$
By choosing~$M \ge M_1 := \Tail_{s, p}(u_+; R) + R^{(\lambda + s p)/p} d$, we are finally led to the estimate
$$
\varphi_{i + 1} \le \frac{C \, 2^{(n + 3 p) i}}{M^{s p^2 / n} R^{s p}} \, \varphi_i^{1 + \frac{s p}{n}}.
$$
Thanks to a standard numerical lemma (e.g.,~\cite[Lemma~7.1]{Giu03}), we conclude that~$\varphi_i$ converges to~$0$, provided~$M$ is greater than the constant~$M_2 := C' R^{- n / p} \| u_+ \|_{L^p(B_{2 R})}$ with~$C' \ge 1$ large enough, in dependence of~$n$,~$s$,~$p$, and~$H$ only. This gives~\eqref{ulocbound}.
\end{proof}

Following the theory of classical De Giorgi classes, the natural next step would be to understand whether the functions of~$\wDG$ are H\"older continuous. It turns out that this is not the case, at least when~$s p < 1$. This is a consequence of an explicit one-dimensional example that we will present in Appendix~\ref{charexapp}.

\begin{question}
Is it true that functions in~$\wDG$ are H\"older continuous, when~$sp \ge 1$?
\end{question}

In order to extend the H\"older regularity estimates that are true for classical~De~Giorgi classes, we are thus forced in general to consider a strict subset of~$\wDG$. To this aim, we propose the following definition.

\begin{definition}[\bf Fractional De Giorgi class $\DG$] \label{DGdef}
Let~$d, \lambda \ge 0$ and~$H \ge 1$. A function~$u \in L^{p - 1}_s(\R^n)$ with~$u|_\Omega \in W^{s, p}(\Omega)$ belongs to~$u \in \DGpm(\Omega; d, H, \lambda)$ if
\begin{equation} \label{DG+-def}
\begin{aligned}
& [(u - k)_\pm]_{W^{s, p}(B_{r}(x_0))}^p + \int_{B_{r}(x_0)} (u(x) - k)_\pm \left\{ \int_{\R^n} \frac{ (u(y) - k)_\mp^{p - 1}}{|x - y|^{n + s p}} \, dy \right\} dx \\
& \hspace{30pt} \le H \, \Bigg\{ R^\lambda d^p |A^\pm(k, x_0, R)| + \frac{R^{(1 - s) p}}{(R - r)^p} \| (u - k)_\pm \|_{L^p(B_R(x_0))}^p \\
& \hspace{30pt} \quad + \frac{R^{n + s p}}{(R - r)^{n + s p}} \| (u - k)_\pm \|_{L^1(B_R(x_0))} \overline{\Tail}_{s, p}((u - k)_\pm; x_0, r)^{p - 1} \Bigg\}
\end{aligned}
\end{equation}
holds for every point~$x_0 \in \Omega$, radii~$0 < r < R < \dist(x_0, \partial \Omega)$, and level~$k \in \R$. We then set~$\DG(\Omega; d, H, \lambda) := \DGp(\Omega; d, H, \lambda) \cap \DGm(\Omega; d, H, \lambda)$.

\end{definition}

We will call~$\DG$ a \emph{strong fractional~De~Giorgi class} or, simply, a \emph{fractional~De~Giorgi class}. It is clear that~$\DG$ is a subset of~$\wDG$. The difference between the two classes lies in the fact that the elements of~$\DG$ satisfy the stronger Caccioppoli-type inequality~\eqref{DG+-def}, which improves~\eqref{wDG+-def} via the presence of an additional summand on its left-hand side. We remark that the specific structure of this term can be partially altered without totally spoiling the results that will follow in the remainder of the section. For instance, if one replaces it with the smaller (and, perhaps, more natural) quantity
$$
\int_{B_{r}(x_0)} \int_{B_r(x_0)} \frac{ (u(x) - k)_\pm (u(y) - k)_\mp^{p - 1}}{|x - y|^{n + s p}} \, dx dy,
$$
all future statements will still hold, apart from the Harnack inequality of Theorem~\ref{DGharthm}.

Though more artificial than~\eqref{wDG+-def}, inequality~\eqref{DG+-def} is still satisfied by solutions of problems involving energies and operators of fractional order, as we will see in Sections~\ref{minsec} and~\ref{solsec} In addition, it turns out that definition~\eqref{DG+-def} is strong enough to guarantee the H\"older continuity of the functions that satisfy it. This has been first realized by Caffarelli, Chan \& Vasseur~\cite{CCV11} for a similar inequality in the context of nonlocal parabolic equations.

Here is our H\"older regularity result for functions in~$\DG$.

\begin{theorem}[\bf H\"older continuity of functions in $\DG$] \label{ulocHolderthm}
Let~$u \in \DG(\Omega; d, H, \lambda)$ for some~$d, \lambda \ge 0$ and~$H \ge 1$. Then,~$u \in C^\alpha_\loc(\Omega)$ for some~$\alpha \in (0, 1)$ and there exists a constant~$C \ge 1$ such that
$$
[ u ]_{C^\alpha(B_R(x_0))} \le \frac{C}{R^\alpha} \left( \| u \|_{L^\infty(B_{2 R}(x_0))} + \Tail_{s, p}(u; x_0, 2 R) + R^{\frac{\lambda + s p}{p}} d \right)
$$
for every~$x_0 \in \Omega$ and~$0 < R < \dist \left( x_0, \partial \Omega \right) / 2$. The constants~$\alpha$ and~$C$ depend only on~$n$,~$s$,~$p$,~$H$, and~$\lambda$.
\end{theorem}

Theorem~\ref{ulocHolderthm} can be proved via an inductive argument based on subsequent applications of a suitable \emph{growth lemma} at smaller and smaller scales. This method goes back to~De~Giorgi~\cite{DeG57} and our implementation of it in this framework follows rather closely the approaches of Silvestre~\cite{Sil06} and~Kassmann~\cite{Kas09,Kas11}. We omit further details, that can be found in the proof of~\cite[Theorem~6.4]{C17}.

The statement of the growth lemma is as follows.


\begin{lemma} \label{growthlem}
Let~$u \in \DGm(B_{4 R}; d, H, \lambda)$ for some~$d, \lambda \ge 0$,~$H \ge 1$, and~$R > 0$. Assume that
$$
u \ge 0 \quad \mbox{in } B_{4 R}
$$
and
$$
\left| B_{2 R} \cap \{ u \ge 1 \} \right| \ge \frac{1}{2} \left| B_{2 R} \right|.
$$
There exists a constant~$\delta \in (0, 1/8]$, depending only on~$n$,~$s$,~$p$,~$H$, and~$\lambda$, such that, if
$$
R^{\frac{\lambda + s p}{p}} d + \Tail_{s,p}(u_-; 4 R) \le \delta,
$$
then
$$
u \ge \delta \quad \mbox{in } B_R.
$$
\end{lemma}

We split the proof of Lemma~\ref{growthlem} into two sublemmata. Interestingly, the first one only relies on the weaker Caccioppoli inequality~\eqref{wDG+-def} and is therefore valid for all functions in~$\wDGm$.


\begin{lemma} \label{stepIIlem}
Let~$u \in \wDGm(B_4; d, H, \lambda)$ for some~$d, \lambda \ge 0$ and~$H \ge 1$. There exists a constant~$\tau \in (0, 2^{- n - 1}]$, depending only on~$n$,~$s$,~$p$,~$H$, and~$\lambda$, such that if~$u \ge 0$ in~$B_2$,
\begin{equation} \label{denstausmall}
\left| B_2 \cap \{ u < 2 \delta \} \right| \le \tau \left| B_2 \right|,
\end{equation}
and
\begin{equation} \label{deltasmall1}
d + \Tail_{s,p}(u_-; 2) \le \delta,
\end{equation}
for some~$\delta \in (0, 1/2]$, then
\begin{equation} \label{ugedelta1}
u \ge \delta \quad \mbox{in } B_1.
\end{equation}
\end{lemma}

\begin{proof}
Let~$\delta \le h < k \le 2 \delta$ and~$1 \le \rho < r \le 2$ be fixed, and~$\tau \in (0, 2^{- n - 1}]$ to be later taken small. Setting~$z_k := (u - k)_-$, we first observe that, by~\eqref{denstausmall} and the fact that~$\tau \le 2^{- n - 1}$, it holds
$$
\left| B_\rho \cap \{ z_k = 0 \} \right| = \left| B_\rho \setminus \{ u < k \} \right| \ge | B_\rho | - \left| B_\rho \cap \{ u < 2 \delta\} \right| \ge | B_\rho | - \tau | B_2 | \ge |B_\rho| / 2.
$$
By this, we may apply the fractional Sobolev inequality for functions that vanish over a set with positive density (see, e.g.,~\cite[Corollary~4.9]{C17}) and get that
\begin{align*}
(k - h) |A^-(h, \rho)|^{\frac{2 n - s}{2n}} & \le \left( \int_{B_{\rho}} z_k(x)^{\frac{2 n}{2 n - s}} \, dx \right)^{\frac{2n - s}{2n}} \le C \int_{A^-(k, \rho)} \int_{B_\rho} \frac{|z_k(x) - z_k(y)|}{|x - y|^{n + \frac{s}{2}}} \, dx dy \\
& \le C |A^-(k, \rho)|^\frac{p - 1}{p} [z_k]_{W^{s, p}(B_\rho)},
\end{align*}
for some constant~$C \ge 1$ depending only on~$n$,~$s$, and~$p$. Note that the last estimate follows by H\"older's inequality---see, e.g.,~\cite[Lemma~4.6]{C17} for the detailed computation. Taking advantage of~\eqref{wDG+-def}, we further obtain that
\begin{align*}
& (k - h)^p |A^-(h, \rho)|^{\frac{(2 n - s) p}{2n}} \\
& \hspace{30pt} \le C |A^-(k, \rho)|^{p - 1}
\left\{ d^p |A^-(k, r)| + \frac{\| z_k \|_{L^p(B_r)}^p}{(r - \rho)^p} + \frac{\| z_k \|_{L^1(B_r)} \overline{\Tail}_{s, p}(z_k; \rho)^{p - 1}}{(r - \rho)^{n + s p}} \right\},
\end{align*}
where~$C$ may now also depend on~$H$ and~$\lambda$. Now, thanks to assumption~\eqref{deltasmall1}, the non-negativity of~$u$ in~$B_2$, and the fact that~$\delta \le k$, from the previous inequality we easily deduce that
$$
|A^-(h, \rho)|^{\frac{2n - s}{2n}} \le \frac{C \, k}{(r - \rho)^{\frac{n + p}{p}} (k - h)} \, |A^-(k, r)|.
$$
By evaluating this inequality along two sequences of radii~$\{ \rho_i \}$ and of levels~$\{ k_i \}$---exponentially decreasing from~$2$ to~$1$ and from~$2 \delta$ to~$\delta$, respectively---and arguing as in the last part of the proof of Proposition~\ref{ulocboundprop}, we are led to the conclusion~\eqref{ugedelta1}, provided~$\tau$ is chosen sufficiently small.
\end{proof}

The second step in the proof of Lemma~\ref{growthlem} is represented by the next result. Unlike Lemma~\ref{stepIIlem}, it heavily relies on the presence of the second term on the left-hand side of~\eqref{DG+-def} and, therefore, it only holds true for functions in the smaller class~$\DGm$.

\begin{lemma} \label{stepIlem}
Let~$u \in \DGm(B_4; d, H, \lambda)$ for some~$d, \lambda \ge 0$,~$H \ge 1$. For every~$\tau \in (0, 1)$, there exists~$\delta \in (0, 1/8]$, depending only on~$n$,~$s$,~$p$,~$H$,~$\lambda$, and~$\tau$, such that if~$u \ge 0$ in~$B_4$,
\begin{equation} \label{uge1dens1}
\left| B_2 \cap \{ u \ge 1 \} \right| \ge \frac{1}{2} \left| B_2 \right|,
\end{equation}
and
$$
d + \Tail_{s,p}(u_-; 4) \le \delta,
$$
then
\begin{equation} \label{u<2deltadens}
\left| B_2 \cap \{ u < 2 \delta \} \right| \le \tau \left| B_2 \right|.
\end{equation}
\end{lemma}

\begin{proof}
We apply~\eqref{DG+-def} with~$x_0 = 0$,~$r = 2$,~$R = 3$, and~$k = 4 \delta$, for some~$\delta \in (0, 1/8]$ to be determined. By arguing as in the last part of the proof of Lemma~\ref{stepIIlem}, it is easy to see that the right-hand side of~\eqref{DG+-def} can be controlled from above by~$C \delta^p$, for some constant~$C \ge 1$ depending only on~$n$,~$s$,~$p$,~$H$, and~$\lambda$. On the other hand, its left-hand side---and, in particular, its second summand---is larger than
$$
\int_{B_2} (4 \delta - u(x))_+ \left\{ \int_{B_2} \frac{ (u(y) - 4 \delta)_+^{p - 1}}{|x - y|^{n + s p}} \, dy \right\} dx \ge c \, \delta \left| B_2 \cap \{ u < 2 \delta \} \right| \left| B_2 \cap \{ u \ge 1 \} \right|,
$$
for some constant~$c \in (0, 1)$ depending only on~$n$,~$s$, and~$p$. By combining together these two facts and recalling hypothesis~\eqref{uge1dens1}, we deduce that~$\left| B_2 \cap \{ u < 2 \delta \} \right| \le C \delta^{p - 1}$, from which~\eqref{u<2deltadens} readily follows, provided~$\delta$ is small enough.
\end{proof}

Since, by scaling, we may reduce ourselves to the case of~$R = 1$, it is clear that the joint application of Lemmata~\ref{stepIIlem} and~\ref{stepIlem} leads to Lemma~\ref{growthlem}.

The growth lemma is the key ingredient of another important result valid for the elements of the class~$\DG$: the Harnack inequality.

\begin{theorem}[\bfseries Harnack inequality for~$\DG$] \label{DGharthm}
Let~$u \in \DG(\Omega; d, H, \lambda)$ for some constants~$d, \lambda \ge 0$ and~$H \ge 1$. There exists a constant~$C \ge 1$, depending on~$n$,~$s$,~$p$,~$\lambda$, and~$H$, such that, if~$u \ge 0$ in~$\Omega$, then
\begin{equation} \label{DGhar}
\sup_{B_R(x_0)} u + \Tail_{s, p}(u_+; x_0, R) \le C \left( \inf_{B_R(x_0)} u + \Tail_{s, p}(u_-; x_0, R) + R^{\frac{\lambda + s p}{p}} d \right)
\end{equation}
for every~$x_0 \in \Omega$ and~$0 < R < \dist (x_0, \partial \Omega) / 2$.
\end{theorem}

Notice the presence of tail terms on both sides of the inequality. The one on the right accounts for the possible negativity of~$u$ outside of~$\Omega$ and cannot be removed, as it was noticed by Kassmann~\cite{Kas11} for~$s$-harmonic functions. Conversely, the one on the left makes the inequality stronger. To the best of our knowledge, the possibility of including such a term was first realized by Ros-Oton \& Serra~\cite{RS16} in the case of the \emph{weak} Harnack inequality (see the forthcoming Theorem~\ref{DGweakHarthm}) for supersolutions of fully nonlinear nonlocal equations. We also mention the recent~\cite{CC17}, by Cabr\'e and the author of this note, where the presence of this extra term is crucially exploited to obtain a gradient bound for nonlocal minimal graphs.

As for the H\"older continuity result, Theorem~\ref{DGharthm} does not hold for the elements of the larger class~$\wDG$ when~$s p < 1$, in view of Proposition~\ref{chiinwDGprop}.

To obtain Theorem~\ref{DGharthm}, we first establish the aforementioned weak Harnack inequality for the class~$\DGm$.

\begin{theorem}[\bfseries Weak Harnack inequality for~$\DGm$] \label{DGweakHarthm}
Let~$u \in \DGm(\Omega; d, H, \lambda)$ for some~$d, \lambda \ge 0$ and~$H \ge 1$. There exist an exponent~$\varepsilon > 0$ and a constant~$C \ge 1$, both depending only on~$n$,~$s$,~$p$,~$\lambda$, and~$H$, such that, if~$u \ge 0$ in~$\Omega$, then
\begin{equation} \label{DGweakHar}
\left( \dashint_{B_R(x_0)} u(x)^\varepsilon \, dx \right)^{\frac{1}{\varepsilon}} \le C \left( \inf_{B_R(x_0)} u + \Tail_{s, p}(u_-; x_0, R) + R^{\frac{\lambda + s p}{p}} d \right)
\end{equation}
for every~$x_0 \in \Omega$ and~$0 < R < \dist (x_0, \partial \Omega) / 2$.
\end{theorem}

For the sake of conciseness, we do not include here the proof of Theorem~\ref{DGweakHarthm}, which essentially relies on a scaled version of Lemma~\ref{growthlem} along with a Krylov-Safonov-type covering lemma. This argument is similar to the one developed in~\cite[Section~3]{DT84} for classical~De~Giorgi classes and can be found in~\cite[Subsection~6.4]{C17}.

\begin{question}
Is it possible to establish a weak Harnack inequality for functions in~$\DG_-$ identical in structure to those of~\cite[Theorem~2.2]{RS16} and~\cite[Theorem~1.6]{CC17}? That is, does~\eqref{DGweakHar} hold with~$\varepsilon = 1$ and with the additional term~$\Tail_{s, p}(u_+; x_0, R)$ added on the left-hand side, such as in~\eqref{DGhar}?
\end{question}

The full Harnack inequality of Theorem~\ref{DGharthm} follows in an almost straightforward way by putting together Theorem~\ref{DGweakHarthm}, (a slightly improved version of) Proposition~\ref{ulocboundprop}, and the next result---again, see~\cite[Subsection~6.4]{C17} for all the details of this argument.

\begin{lemma} \label{taillem}
Let~$u \in \DGm(B_R; d, H, \lambda)$ for some~$d, \lambda \ge 0$,~$H \ge 1$, and~$R > 0$. There is a constant~$C \ge 1$, depending only on~$n$,~$s$,~$p$, and~$H$, such that, if~$u \ge 0$ in~$B_R$, then
$$
\Tail_{s, p}(u_+; R) \le C \left( \sup_{B_R} u + \Tail_{s, p}(u_-; R) + R^{\frac{\lambda + s p}{p}} d \right).
$$
\end{lemma}
\begin{proof}
It suffices to apply inequality~\eqref{DG+-def} with~$x_0 = 0$,~$r = R/2$, and~$k = 2 M$, where we set~$M := \sup_{B_R} u + R^{(\lambda + s p ) / p} d$. On the one hand, it is not hard to see that
\begin{align*}
& \int_{B_{R / 2}} (u(x) - 2 M)_- \left\{ \int_{\R^n} \frac{ (u(y) - 2 M)_+^{p - 1}}{|x - y|^{n + s p}} \, dy \right\} dx \\
& \hspace{30pt} \ge \frac{M R^n}{C} \int_{\R^n \setminus B_R} \frac{ (u(y) - 2 M)_+^{p - 1}}{|y|^{n + s p}} \, dy \ge R^{n - s p} \left\{ \frac{M}{C} \Tail_{s, p}(u_+; R)^{p - 1} - C M^p \right\},
\end{align*}
for some constant~$C \ge 1$ depending only on~$n$,~$s$, and~$p$. On the other hand, the right-hand side of~\eqref{DG+-def} is controlled by~$C R^{n - s p} \left\{ M^p + M \Tail_{s, p}(u_-; R)^{p - 1} \right\}$, with~$C$ now depending on~$H$ as well. The lemma then plainly follows by comparing these two expressions.
\end{proof}

We conclude the section with a comment on the stability of the results that we just presented in the limit as~$s \nearrow 1$.

Essentially all the estimates that we obtained can be made uniform with respect to this limit---that is, the constants that govern them can be chosen to be independent of~$s$, for~$s$ bounded away from zero---, provided a couple of changes in the definitions of fractional~De~Giorgi classes are carried out: one needs to replace~$H$ with~$H / (1 - s)$ in both~\eqref{wDG+-def} and~\eqref{DG+-def}, and to correct the definition of the tail term by adding the factor~$(1 - s)$ in front of the integral that appears, within round brackets, on the right-hand side of~\eqref{taildef}. After these modifications, all results are uniform as~$s \nearrow 1$ and coherent with those that are known for classical~De~Giorgi classes. See~\cite{C17} for the precise statements.

Such uniformity can be achieved mostly by keeping track of the dependence in~$s$ of all the constants involved in the various results. Everything goes through with little effort besides one point: the behavior of the constant~$\delta$ in Lemma~\ref{stepIlem}. As can be easily checked, the proof of Lemma~\ref{stepIlem} is based exclusively on the estimate for the second term on the left-hand side of~\eqref{DG+-def}, a purely nonlocal quantity that, when multiplied by~$(1 - s)$, vanishes in the limit as~$s \nearrow 1$. As a result, the proof of Lemma~\ref{stepIlem} is not uniform in~$s$ as it is. To make it uniform, one can interpolate such proof with an argument closer in spirit to one that leads to the growth lemma for classical~De~Giorgi classes, such as~\cite[Lemma~7.5]{Giu03}.

A key element of the proof of this classical result is an isoperimetric-type inequality for the level sets of functions in~$W^{1, p}$ due to~De~Giorgi~\cite{DeG57}---see, e.g.,~\cite{CV12} for its statement when~$p = 2$ and~\cite[Lemma~5.2]{C17} for the general case. Next is a partial extension of this inequality to the fractional Sobolev space~$W^{s, p}$, when~$s$ is close to~$1$.

\begin{proposition} \label{sDGlemprop}
Let~$n \ge 2$,~$M > 0$, and~$\gamma \in (0, 1)$. There exist two constants~$\bar{s} \in (0, 1)$ and~$C > 0$ such that the inequality
\begin{align*}
& \Big\{ |B_1 \cap \{ u \le 0 \}| | B_1 \cap \{ u \ge 1 \} | \Big\}^{\frac{n - 1}{n}}  \le C (1 - s)^{1 / p} [u]_{W^{s, p}(B_1)} |B_1 \cap \{ 0 < u < 1 \}|^{\frac{p - 1}{p}}
\end{align*}
holds true for every~$s \in [\bar{s}, 1)$ and every function~$u \in W^{s, p}(B_1)$ satisfying
\begin{equation} \label{quantass}
\begin{gathered}
\| u \|_{L^p(B_1)}^p + (1 - s) [u]_{W^{s, p}(B_1)}^p \le M, \\
|B_1 \cap \{ u \le 0 \}| \ge \gamma |B_1| \quad \mbox{and} \quad |B_1 \cap \{ u \ge 1 \}| \ge \gamma |B_1|.
\end{gathered}
\end{equation}
The constant~$C$ depends only on~$n$ and~$p$, while~$\bar{s}$ also depends on~$M$ and~$\gamma$.
\end{proposition}

The proof of Proposition~\ref{sDGlemprop} presented in~\cite[Section~5]{C17} is by contradiction and based on a compactness argument that relies on the aforementioned~De~Giorgi's isoperimetric inequality in the Sobolev space~$W^{1, p}$. As a result, the optimal value of~$\bar{s}$ is unknown, as well as its possible independence from~$M$ and~$\gamma$. However, it necessarily holds that~$\bar{s} \ge 1/p$, due to the fact that~$\chi_E \in W^{s, p}(B_1)$ for every~$s \in (0, 1/p)$ and every smooth set~$E \subset B_1$.

\begin{question}
Is it possible to obtain an inequality similar to the one of Proposition~\ref{sDGlemprop} for every function of the space~$W^{s, p}(B_1)$, every~$s \in [1/p, 1)$, and without assuming~\eqref{quantass}?
\end{question}

\section{Applications to minimizers of nonlocal functionals} \label{minsec}

In this section we present the main application of fractional~De~Giorgi classes, which ultimately motivates their introduction: the H\"older regularity of minimizers of possibly non-differentiable nonlocal functionals.

Let~$K: \R^n \times \R^n \to \R$ be a non-negative measurable function satisfying
\begin{equation} \label{Ksymm}
K(x, y) = K(y, x) \quad \mbox{for a.e.~} x, y \in \R^n
\end{equation}
and
\begin{equation} \label{Kbounds}
\frac{1}{\Lambda |x - y|^{n + s p}} \le K(x, y) \le \frac{\Lambda}{|x - y|^{n + s p}} \quad \mbox{for a.e.~} x, y \in \R^n, 
\end{equation}
for some constants~$s \in (0, 1)$,~$p > 1$, and~$\Lambda \ge 1$. Let~$F: \Omega \times \R \to \R$ be a Carath\'eodory function and assume that
\begin{equation} \label{Fbounds}
|F(x, u)| \le F_0 \quad \mbox{for a.e.~} x \in \Omega \mbox{ and every } u \in \R,
\end{equation}
for some constant~$F_0 \ge 0$. Associated to these two functions, we consider the energy functional~$\E$, defined on every measurable function~$u: \R^n \to \R$ by
$$
\E(u; \Omega) := \frac{1}{2p} \iint_{\C_\Omega} \left| u(x) - u(y) \right|^p K(x, y) \, dx dy + \int_\Omega F(x, u(x)) \, dx,
$$
where~$\C_{\Omega} := \R^{2 n} \setminus (\R^n \setminus \Omega)^2$. More general kernels~$K$ and unbounded potentials~$F$ (with subcritical growth in~$u$) can also be considered. For simplicity of exposition, here we restrict ourselves to those that are allowed by hypotheses~\eqref{Ksymm}-\eqref{Fbounds}. We refer the interested reader to~\cite{C17} for a broader setting.

The functional~$\E$ has been recently considered by several authors, since it allows to model nonlinear phenomena that occur in the presence of long-range interactions. Here, we are particularly interested in the case when~$F$ is not differentiable (and, perhaps, not even continuous) in the variable~$u$. Examples of such potentials have been considered for instance in~\cite{CRS10b}, with~$F(u) = \chi_{(0, +\infty)}(u)$, and in~\cite{CV17}, with~$F(u)$ comparable to~$|1 - u^2|^d$ for some~$d > 0$.

Notice that, under the sole assumption~\eqref{Fbounds}, the functional~$\E$ is not differentiable and therefore the regularity properties of its minimizers cannot be inferred from a Euler-Lagrange equation. Instead, we will deduce such properties directly from the minimizing inequality, as done in~\cite{GG82} in the case of a local functional.

We now specify the notion of minimizers that we take into account. To this aim, we will say that a function~$u: \R^n \to \R$ belongs to~$\W^{s, p}(\Omega)$ if~$u|_\Omega \in L^p(\Omega)$ and
$$
\iint_{\C_\Omega} \frac{|u(x) - u(y)|^p}{|x - y|^{n + sp}} \, dx dy < +\infty.
$$
By~\eqref{Kbounds} and~\eqref{Fbounds}, this is equivalent to ask that~$u|_\Omega \in L^p(\Omega)$ and~$\E(u; \Omega) < +\infty$.

\begin{definition}
A function~$u \in \W^{s, p}(\Omega)$ is a \emph{superminimizer} of~$\E$ in~$\Omega$ if
\begin{equation} \label{EuleEu+phi}
\E(u; \Omega) \le \E(u + \varphi; \Omega)
\end{equation}
for every non-negative measurable function~$\varphi: \R^n \to \R$ supported inside~$\Omega$. Similarly,~$u$ is a \emph{subminimizer} of~$\E$ in~$\Omega$ if~\eqref{EuleEu+phi} holds true for every non-positive such~$\varphi$. Finally,~$u$ is a \emph{minimizer} of~$\E$ in~$\Omega$ if~\eqref{EuleEu+phi} holds for every measurable~$\varphi: \R^n \to \R$ supported inside~$\Omega$.
\end{definition}

It is not hard to check that~$u$ is a minimizer if and only if it is at the same time a sub- and a superminimizer.

In the following result, we establish that minimizers of the energy functional~$\E$ belong to a fractional~De~Giorgi class.

\begin{theorem} \label{minDeG}
Let~$u \in L^{p - 1}_s(\R^n) \cap \W^{s, p}(\Omega)$. There exists a constant~$H \ge 1$, depending only on~$n$,~$s$,~$p$, and~$\Lambda$, such that:
\begin{enumerate}[label=$(\alph*)$,leftmargin=*]
\item if~$u$ is a superminimizer of~$\E$ in~$\Omega$, then~$u \in \DGm(\Omega; F_0^{1/p}, H, 0)$;
\item if~$u$ is a subminimizer of~$\E$ in~$\Omega$, then~$u \in \DGp(\Omega; F_0^{1/p}, H, 0)$;
\item if~$u$ is a minimizer of~$\E$ in~$\Omega$, then~$u \in \DG(\Omega; F_0^{1/p}, H, 0)$.
\end{enumerate}
\end{theorem}


For the sake of simplicity, we present the proof of Theorem~\ref{minDeG} only for~$p = 2$. With little additional technical effort, the argument can be easily extended to the case of a general~$p > 1$, as shown in the proof of~\cite[Proposition~7.5]{C17}.

\begin{proof}[Proof of Theorem~\ref{minDeG} for~$p = 2$]
We only deal with point~$(a)$, since~$(b)$ is completely analogous. Clearly,~$(c)$ immediately follows from~$(a)$ and~$(b)$.


Let~$x_0 \in \Omega$ and~$0 < r \le \rho < \tau \le R < \dist(x_0, \partial \Omega)$. Up to a translation, we may suppose that~$x_0 = 0$. Let~$\eta \in C^\infty(\R^n)$ be a cutoff function satisfying~$0 \le \eta \le 1$ in~$\R^n$,~$\supp(\eta) \subseteq B_{(\tau + \rho) / 2}$,~$\eta = 1$ in~$B_\rho$, and~$|\nabla \eta| \le 4 / (\tau - \rho)$ in~$\R^n$.

For any fixed~$k \in \R$, let~$w_\pm := (u - k)_\pm$,~$\varphi := \eta w_-$, and choose~$v := u + \varphi$ as a competitor for~$u$ in~\eqref{EuleEu+phi}. It holds
\begin{equation} \label{minineq}
\iint_{\C_{B_\tau}} \Xi(x, y) \, d\mu(x, y) \le 4 \int_{B_\tau} \Big\{ F(x, v(x)) - F(x, u(x)) \Big\} \, dx,
\end{equation}
where~$\Xi(x, y) := \left| u(x) - u(y) \right|^2 - \left| v(x) - v(y) \right|^2$ and~$d\mu(x, y) := K(x, y) \, dx dy$.

Now, on the one hand, by~\eqref{Fbounds} we have that
$$
F(x, v(x)) - F(x, u(x)) = F(x, u(x) + \eta(x) w_-(x)) - F(x, u(x)) \le 2 F_0 \, \chi_{A^-(k, \rho)}(x)
$$
for every~$x \in B_\rho$. By this, we easily obtain an upper bound for the right-hand side of~\eqref{minineq}:
\begin{equation} \label{RHSest}
\int_{B_\tau} \Big\{ F(x, v(x)) - F(x, u(x)) \Big\} \, dx \le 2 F_0 \left| A^-(k, \tau) \right|.
\end{equation}

On the other hand, we estimate the left-hand side of~\eqref{minineq} as follows. Using the definition of~$w_-$ along with Young's inequality, we get that, for every~$(x, y) \in A^-(k) \times A^-(k)$,
\begin{align*}
\Xi(x, y)
& = \left| w_-(x) - w_-(y) \right|^2 - \left| (1 - \eta(x)) \left( w_-(x) - w_-(y) \right) - \left( \eta(x) - \eta(y) \right) w_-(y) \right|^2 \\
& \ge \left( 1 - 2 (1 - \eta(x))^2 \right) \left| w_-(x) - w_-(y) \right|^2 - 2 \left| \eta(x) - \eta(y) \right|^2 w_-(y)^2.
\end{align*}
In particular, when~$x \in A^-(k) \setminus B_\tau$ and~$y \in A^-(k)$, it also holds
\begin{align*}
\Xi(x, y) & = \left| w_-(x) - w_-(y) \right|^2 - \left| \left( w_-(x) - w_-(y) \right) + \eta(y) w_-(y) \right|^2 \\
& \ge - 2 \eta(y) w_-(x) w_-(y).
\end{align*}
For~$(x, y) \in A^-(k) \times \left( \R^n \setminus A^-(k) \right)$ we have
\begin{align*}
\Xi(x, y)
& = \eta(x) w_-(x) \big( w_-(x) + 2 w_+(y) + (1 - \eta(x)) w_-(x) \big) \\
& \ge \eta(x) \left( \vphantom{\big(} \left| w_-(x) - w_-(y) \right|^2 + 2 w_-(x) w_+(y) \right).
\end{align*}
By these inequalities, the fact that~$\Xi(x, y) = 0$ for~$x, y \in \R^n \setminus A^-(k)$, hypotheses~\eqref{Ksymm}-\eqref{Kbounds} on the kernel~$K$, and the properties of the cutoff~$\eta$, it is not hard to conclude that
\begin{align*}
\iint_{\C_{B_\tau}} \Xi(x, y) \, d\mu(x, y) & \ge \frac{1}{C} \left\{ [w_-]_{H^s(B_\rho)}^2 + \int_{B_\rho} w_-(x) \left\{ \int_{\R^n} \frac{w_+(y)}{|x - y|^{n + 2 s}} \, dy \right\} dx \right\} \\
& \quad - C \, \Bigg\{ \iint_{B_\tau^2 \setminus B_\rho^2} \frac{\left| w_-(x) - w_-(y) \right|^2}{|x - y|^{n + 2 s}} \, dx dy \\
& \quad + \frac{R^{2 - 2 s}}{(\tau - \rho)^2} \| w_- \|_{L^2(B_R)}^2 + \frac{R^{n + 2 s}}{(\tau - \rho)^{n + 2 s}} \| w_- \|_{L^1(B_R)} \overline{\Tail}_{s, 2}(w_-; r) \Bigg\}
\end{align*}
for some constant~$C \ge 1$ depending only on~$n$,~$s$,~$p$, and~$\Lambda$.

By putting together the last estimate,~\eqref{RHSest}, and~\eqref{minineq}, and applying Widman's hole-filling technique (with respect to the term~$[w_-]_{H^s(B_\rho)}^2$), we obtain
\begin{align*}
& [w_-]_{H^s(B_\rho)}^2 + \int_{B_\rho} w_-(x) \left\{ \int_{\R^n} \frac{w_+(y)}{|x - y|^{n + 2 s}} \, dy \right\} dx \le \gamma \, \Bigg\{ [w_-]_{H^s(B_\tau)}^2 \\
& \hspace{20pt} + F_0 \left| A^-(k, R) \right| + \frac{R^{2 - 2 s}}{(\tau - \rho)^2} \| w_- \|_{L^2(B_R)}^2 + \frac{R^{n + 2 s}}{(\tau - \rho)^{n + 2 s}} \| w_- \|_{L^1(B_R)} \overline{\Tail}_{s, 2}(w_-; r) \Bigg\}
\end{align*}
for some constant~$\gamma \in (0, 1)$ depending only on~$n$,~$s$,~$p$, and~$\Lambda$. From this and a simple iteration lemma (see, e.g.,~\cite[Lemma~4.11]{C17}) it follows that~$u \in \DG_-(\Omega; F_0^{1/p}, H, 0)$ for some~$H \ge 1$ depending only on~$n$,~$s$,~$p$, and~$\Lambda$.
\end{proof}

By combining this result with Theorems~\ref{ulocboundthm} and~\ref{ulocHolderthm} of Section~\ref{DeGsec}, we deduce the H\"older continuity of minimizers of~$\E$.

\begin{corollary}[\bf H\"older continuity of minimizers] \label{minHoldcor}
Let~$u \in L^{p - 1}_s(\R^n) \cap \W^{s, p}(\Omega)$ be a minimizer of~$\E$ in~$\Omega$. Then,~$u \in C^\alpha_\loc(\Omega)$ for some exponent~$\alpha \in (0, 1)$ and there exists a constant~$C \ge 1$ such that
$$
\| u \|_{L^\infty(B_R(x_0))} + R^\alpha [ u ]_{C^\alpha(B_R(x_0))} \le C \left( \frac{\| u \|_{L^p(B_{2 R}(x_0))}}{R^{n/p}} + \Tail_{s, p}(u; x_0, R) + R^s F_0^{1/p} \right)
$$
for every point~$x_0 \in \Omega$ and radius~$0 < R < \dist (x_0, \partial \Omega) / 2$. The constants~$\alpha$ and~$C$ depend only on~$n$,~$s$,~$p$, and~$\Lambda$.
\end{corollary}

Similarly, by Theorems~\ref{DGharthm},~\ref{DGweakHarthm}, and~\ref{minDeG}, non-negative minimizers of~$\E$ satisfies the following Harnack-type inequalities.

\begin{corollary}[\bf Harnack inequalities for minimizers] Let~$u \in L^{p - 1}_s(\R^n) \cap \W^{s, p}(\Omega)$ with~$u \ge 0$ in~$\Omega$. The following statements hold true:
\begin{enumerate}[label=$(\alph*)$,leftmargin=*]
\item if~$u$ is a superminimizer of~$\E$ in~$\Omega$, then there exist an exponent~$\varepsilon > 0$ and a constant~$C \ge 1$, both depending only on~$n$,~$s$,~$p$, and~$\Lambda$, such that
$$
\left( \dashint_{B_R(x_0)} u(x)^\varepsilon \, dx \right)^{\frac{1}{\varepsilon}} \le C \left( \inf_{B_R(x_0)} u + \Tail_{s, p}(u_-; x_0, R) + R^s F_0^{1/p} \right)
$$
for every~$x_0 \in \Omega$ and~$0 < R < \dist (x_0, \partial \Omega) / 2$;
\item if~$u$ is a minimizer of~$\E$ in~$\Omega$, then there exists a constant~$C \ge 1$, only depending on~$n$,~$s$,~$p$, and~$\Lambda$, such that
$$
\sup_{B_R(x_0)} u + \Tail_{s, p}(u_+; x_0, R) \le C \left( \inf_{B_R(x_0)} u + \Tail_{s, p}(u_-; x_0, R) + R^s F_0^{1/p} \right)
$$
for every~$x_0 \in \Omega$ and~$0 < R < \dist (x_0, \partial \Omega) / 2$.
\end{enumerate}
\end{corollary}

\section{Applications to solutions of nonlocal equations} \label{solsec}

Another application of fractional De Giorgi classes is represented by the regularity results that will be discussed in this section, concerning weak solutions of equations driven by nonlocal operators.

Let~$K$ be a kernel satisfying~\eqref{Ksymm} and~\eqref{Kbounds}, for some~$s \in (0, 1)$,~$p > 1$, and~$\Lambda \ge 1$. We introduce the nonlinear and nonlocal operator~$\L = \L_{K, p}$ as formally defined on a measurable function~$u: \R^n \to \R$ and at a point~$x \in \R^n$ by
\begin{align*}
\L u(x) := & \,\, \PV \int_{\R^n} |u(x) - u(y)|^{p - 2} (u(x) - u(y)) K(x, y) \, dy \\
= & \,\, \lim_{\delta \searrow 0} \int_{\R^n \setminus B_\delta(x)} |u(x) - u(y)|^{p - 2} (u(x) - u(y)) K(x, y) \, dy.
\end{align*}
Also, let~$f \in L^\infty(\Omega)$ and~$f_0 \ge \| f \|_{L^\infty(\Omega)}$ be given.

Throughout the section, we will consider (sub-/super-)solutions of the equation
\begin{equation} \label{eq}
\L u = f \quad \mbox{in } \Omega,
\end{equation}
defined in the following weak sense.

\begin{definition}
A function~$u \in \W^{s, p}(\Omega)$ is a \emph{weak supersolution} of~\eqref{eq} if
$$
\frac{1}{2} \int_{\R^n} \int_{\R^n} |u(x) - u(y)|^{p - 2} (u(x) - u(y)) (\varphi(x) - \varphi(y)) K(x, y) \, dx dy \le \int_{\R^n} f(x) \varphi(x) \, dx
$$
for every non-negative function~$\varphi \in W^{s, p}(\R^n)$ supported inside~$\Omega$. Conversely,~$u$ is a \emph{weak subsolution} of~\eqref{eq} if the reverse inequality holds for every such~$\varphi$. Finally,~$u$ is a \emph{weak solution} of~\eqref{eq} if
$$
\frac{1}{2} \int_{\R^n} \int_{\R^n} |u(x) - u(y)|^{p - 2} (u(x) - u(y)) (\varphi(x) - \varphi(y)) K(x, y) \, dx dy = \int_{\R^n} f(x) \varphi(x) \, dx
$$
for every~$\varphi \in W^{s, p}(\R^n)$ supported inside~$\Omega$.
\end{definition}

Note that, if~$F = F(x, u)$ is a differentiable function in the variable~$u$, the minimizers of the energy~$\E$ considered in Section~\ref{minsec} are weak solutions of~\eqref{eq}, with~$f = - F_u(\cdot, u)$. As for those minimizers, solutions of equations driven by the operator~$\L$ are contained in a fractional~De~Giorgi class. This is the content of the next result.

\begin{theorem} \label{solDeG}
Let~$u \in L^{p - 1}_s(\R^n) \cap \W^{s, p}(\Omega)$. There exists a constant~$H \ge 1$, depending only on~$n$,~$s$,~$p$, and~$\Lambda$, such that:
\begin{enumerate}[label=$(\alph*)$,leftmargin=*]
\item if~$u$ is a weak supersolution of~\eqref{eq}, then~$u \in \DGm(\Omega; f_0^{1/(p - 1)}, H, s p / (p - 1))$;
\item if~$u$ is a weak subsolution of~\eqref{eq}, then~$u \in \DGp(\Omega; f_0^{1/(p - 1)}, H, s p / (p - 1))$;
\item if~$u$ is a weak solution of~\eqref{eq}, then~$u \in \DG(\Omega; f_0^{1/(p - 1)}, H, s p / (p - 1))$;
\end{enumerate}
\end{theorem}

Theorem~\ref{solDeG} can be proved through a strategy analogous to Theorem~\ref{minDeG}. We omit the details for brevity and refer the interested reader to~\cite[Section~8]{C17}.

By putting together this result with Theorems~\ref{ulocboundthm} and~\ref{ulocHolderthm}, we are able to deduce the H\"older regularity of weak solutions to~\eqref{eq}.

\begin{corollary}[\bf H\"older continuity of solutions] \label{solHoldcor}
Let~$u \in L^{p - 1}_s(\R^n) \cap \W^{s, p}(\Omega)$ be a weak solution of~\eqref{eq}. Then,~$u \in C^\alpha_\loc(\Omega)$ for some exponent~$\alpha \in (0, 1)$ and there exists a constant~$C \ge 1$ such that
$$
\| u \|_{L^\infty(B_R(x_0))} + R^\alpha [ u ]_{C^\alpha(B_R(x_0))} \le C \left( \frac{\| u \|_{L^p(B_{2 R}(x_0))}}{R^{n / p}} + \Tail_{s, p}(u; x_0, R) + R^{\frac{sp}{p - 1}} f_0^{1/(p - 1)} \right)
$$
for every~$x_0 \in \Omega$ and every~$0 < R < \dist (x_0, \partial \Omega) / 2$. The constants~$\alpha$ and~$C$ only depend on~$n$,~$s$,~$p$, and~$\Lambda$.
\end{corollary}

When~$p = 2$, the~$C^\alpha$ character of solutions to~\eqref{eq} is well-known---see, e.g., Silvestre~\cite{Sil06} and Kassmann~\cite{Kas09}. For a general~$p > 1$, such regularity has been obtained by~Di~Castro, Kuusi \& Palatucci~\cite{DKP16} in the case of~$\L$-harmonic functions, i.e., when~$f \equiv 0$. To the best of our knowledge, Corollary~\ref{solHoldcor}---appeared in~\cite{C17} as Theorem~2.4---is the first result establishing H\"older estimates for solutions of~\eqref{eq} when~$p \ne 2$ and in the presence of a non-vanishing right-hand side~$f$. See also the very recent~\cite{BLS17} for almost sharp results when~$p \ge 2$ and~$s < (p - 1)/p$.

Taking advantage of Theorems~\ref{DGharthm} and~\ref{DGweakHarthm}, we also have the following Harnack inequalities.

\begin{corollary}[\bf Harnack inequalities for solutions]
Let~$u \in L^{p - 1}_s(\R^n) \cap \W^{s, p}(\Omega)$ with~$u \ge 0$ in~$\Omega$. The following statements hold true:
\begin{enumerate}[label=$(\alph*)$,leftmargin=*]
\item if~$u$ is a weak supersolution of~\eqref{eq}, then there exist an exponent~$\varepsilon > 0$ and a constant~$C \ge 1$, both depending only on~$n$,~$s$,~$p$, and~$\Lambda$, such that
$$
\left( \dashint_{B_R(x_0)} u(x)^\varepsilon \, dx \right)^{\frac{1}{\varepsilon}} \le C \left( \inf_{B_R(x_0)} u + \Tail_{s, p}(u_-; x_0, R) + R^{\frac{s p}{p - 1}} f_0^{1/(p - 1)} \right)
$$
for every~$x_0 \in \Omega$ and~$0 < R < \dist (x_0, \partial \Omega) / 2$;
\item if~$u$ is a weak solution of~\eqref{eq}, then there exists a constant~$C \ge 1$, only depending on~$n$,~$s$,~$p$, and~$\Lambda$, such that
$$
\sup_{B_R(x_0)} u + \Tail_{s, p}(u_+; x_0, R) \le C \left( \inf_{B_R(x_0)} u + \Tail_{s, p}(u_-; x_0, R) + R^{\frac{s p}{p - 1}} f_0^{1/(p - 1)} \right)
$$
for every~$x_0 \in \Omega$ and~$0 < R < \dist (x_0, \partial \Omega) / 2$.
\end{enumerate}
\end{corollary}

Similar Harnack inequalities appeared in~\cite{Kas11}, for~$p = 2$, and in~\cite{DKP14}, with a general~$p > 1$ but with~$f = 0$.

\appendix

\section{An explicit example} \label{charexapp}

It is easy to see that the characteristic function of a sufficiently smooth subset~$E$ of~$\R^n$ is contained in the fractional Sobolev space~$W^{s, p}$, provided~$sp < 1$. In this appendix we show that, in dimension~$n = 1$ and under this assumption on~$s$ and~$p$, a step function may also belong to a \emph{weak} fractional De Giorgi class~$\wDG$---but never to a \emph{strong} class~$\DG$. From this, it follows that the~$C^\alpha$ estimates of Theorem~\ref{ulocHolderthm} and the Harnack inequality of Theorem~\ref{DGharthm}---both valid for the elements of the smaller class~$\DG$---cannot be extended to~$\wDG$.

\begin{proposition} \label{chiinwDGprop}
Let~$n = 1$ and~$s p < 1$. Then,
\begin{equation} \label{chiinwDG}
\chi_{(0, +\infty)} \in \wDG((-1, 1); 0, H, 0)
\end{equation}
for some constant~$H \ge 1$. Furthermore,
\begin{equation} \label{chinotinDG}
\chi_{(0, +\infty)} \notin \DGm((-1, 1); d, H, \lambda)
\end{equation}
for every~$d, \lambda \ge 0$ and~$H \ge 1$.
\end{proposition}
\begin{proof}
We begin by showing that~\eqref{chiinwDG} holds true. We only check that~$u := \chi_{(0, +\infty)}$ belongs to the class~$\wDGm$, as the verification of its inclusion in~$\wDGp$ is analogous.

Fix any~$x_0 \in (-1, 1)$,~$0 < r < R < 1 - |x_0|$, and~$k \in \R$. In order to check the validity of the inequality defining~$\wDGm$, we clearly can restrict ourselves to considering the case of~$k > 0$, since otherwise~$(u - k)_- \equiv 0$. For shortness, we only deal with~$k \in (0, 1]$, the case~$k > 1$ being similar. We first estimate from above the left-hand side of~\eqref{wDG+-def}:
\begin{equation} \label{LHSDeG}
\begin{aligned}
[(u - k)_-]_{W^{s, p}((x_0 - r, x_0 + r))}^p & = \int_{x_0 - r}^{x_0 + r} \int_{x_0 - r}^{x_0 + r} \frac{|(u(x) - k)_- - (u(y) - k)_-|^p}{|x - y|^{1 + s p}} \, dx dy \\
& = 2 k^p \chi_{(|x_0|, +\infty)}(r) \int_{x_0 - r}^0 \int_0^{x_0 + r} \frac{dx dy}{|x - y|^{1 + s p}} \\
& \le \frac{2 (r - |x_0|)_+^{1 - s p} k^p}{s p (1 - s p)}.
\end{aligned}
\end{equation}
In view of this, it suffices to estimate from below the right-hand side of~\eqref{wDG+-def} when~$r > |x_0|$. In this case, also~$R > |x_0|$ and therefore such right-hand side is larger than
\begin{align*}
H \frac{R^{(1 - s) p}}{(R - r)^p} \| (u - k)_- \|_{L^p(x_0 - R, x_0 + R)}^p = H \frac{R^{(1 - s) p}}{(R - r)^p} k^p \int_{x_0 - R}^0 dx \ge H (R - |x_0|)^{1 - s p} k^p.
\end{align*}
As~$R > r$, the latter quantity controls the one appearing on the last line of~\eqref{LHSDeG}, provided~$H$ is sufficiently large (in dependence of~$s$ and~$p$ only). Consequently,~$u$ belongs to the class~$\wDGm((-1, 1); 0, H, 0)$.

We now turn our attention to~\eqref{chinotinDG}. We point out that, arguing by contradiction, its validity could be inferred from Theorem~\ref{DGweakHarthm}. Nevertheless, we present here a proof of it based on a direct computation, for we show that inequality~\eqref{DG+-def} does not hold when~$x_0 = 0$ and~$R = 2 r$, with~$k, r > 0$ suitably small. Indeed, under these assumptions the left-hand side of~\eqref{DG+-def} is larger than
\begin{align*}
\int_{-r}^{r} (u(x) - k)_- \left\{ \int_\R \frac{ (u(y) - k)_+^{p - 1}}{|x - y|^{1 + s p}} \, dy \right\} dx & \ge \int_{- r}^{0} k \left\{ \int_{0}^{x + r} \frac{(1 - k)^{p - 1}}{(y - x)^{1 + s p}} \, dy \right\} dx \\
& = \frac{r^{1 - s p} k (1 - k)^{p - 1}}{1 - s p}.
\end{align*}
On the other hand, it is easy to check that the right-hand side of~\eqref{DG+-def} is bounded above by~$C H (r^{1 + \lambda} d^p + r^{1 - s p} k^p)$, for some constant~$C \ge 1$ depending only on~$s$ and~$p$. By taking~$r$ and~$k$ smaller and smaller (but positive), it follows that the latter quantity cannot control the one displayed above, no matter how large~$H$ is. Hence,~\eqref{chinotinDG} holds true.
\end{proof}

\end{document}